\documentclass[11pt]{article}

\usepackage{mathrsfs}
\usepackage{nicefrac}
\usepackage[T1]{fontenc}
\usepackage{latexsym,amssymb,amsmath,amsfonts,amsthm}
\usepackage{graphicx}
\usepackage{caption}
\usepackage[utf8]{inputenc}
\usepackage{authblk}
\usepackage{bm}
\usepackage{accents}
\usepackage{multicol}  
\usepackage{multirow}
\usepackage[pdftex,colorlinks=true,citecolor=blue,linkcolor=blue]{hyperref}

\newtheorem{theorem}{Theorem}[section]
\newtheorem{lemma}{Lemma}[section]
\newtheorem{corollary}{Corollary}[section]
\newtheorem{remark}{Remark}[section]
\newtheorem{example}{Example}[section]
\DeclareMathOperator*{\Null}{null}
\DeclareMathOperator*{\Range}{range}
\DeclareMathOperator*{\rank}{rank}
\DeclareMathOperator*{\diag}{diag}

\makeatletter

\begin{document}

\title{On the Ideal Interpolation Operator in Algebraic Multigrid Methods}
\author{Xuefeng Xu%
\thanks{Corresponding author. LSEC, Institute of Computational Mathematics and Scientific/Engineering Computing, Academy of Mathematics and Systems Science, Chinese Academy of Sciences, Beijing 100190, China; School of Mathematical Sciences, University of Chinese Academy of Sciences, Beijing 100049, China (\texttt{xuxuefeng@lsec.cc.ac.cn}).} 
\quad and \quad 
Chen-Song Zhang%
\thanks{LSEC \& NCMIS, Academy of Mathematics and Systems Science, Chinese Academy of Sciences, Beijing 100190, China; School of Mathematical Sciences, University of Chinese Academy of Sciences, Beijing 100049, China (\texttt{zhangcs@lsec.cc.ac.cn).}}}
\date{\today}
\maketitle

\begin{abstract}
	
Various algebraic multigrid algorithms have been developed for solving problems in scientific and engineering computation over the past decades. They have been shown to be well-suited for solving discretized partial differential equations on unstructured girds in practice. One key ingredient of algebraic multigrid algorithms is a strategy for constructing an effective prolongation operator. Among many questions on constructing a prolongation, an important question is how to evaluate its quality. In this paper, we establish new characterizations (including sufficient condition, necessary condition, and equivalent condition) of the so-called \emph{ideal} interpolation operator. Our result suggests that, compared with common wisdom, one has more room to construct an ideal interpolation, which can provide new insights for designing algebraic multigrid algorithms. Moreover, we derive a new expression for a class of ideal interpolation operators.

\end{abstract}
	
\noindent{\bf Keywords:} algebraic multigrid, ideal interpolation, coarsening

\medskip

\noindent{\bf AMS subject classifications:} 65F10, 65F15, 65N55

\section{Introduction}

Numerical method for solving large-scale systems of equations arising from the discretization of partial differential equations (PDEs) is an active topic of research over the past decades (see, e.g.,~\cite{Saad2003,Hackbusch2016}). Classical iterative methods, like Jacobi and Gauss--Seidel, tend to converge slowly for large-scale problems, because low-frequency (i.e., smooth) error components are attenuated very slowly by these classical methods in general. For the linear systems arising from finite element and finite difference discretizations of elliptic boundary value problems, local relaxation methods are typically effective to eliminate the high-frequency (i.e., oscillatory) error components, while the low-frequency parts cannot be eliminated effectively. The main idea of multigrid methods is to project the error obtained from local relaxation processes onto a coarser grid, which will yield a relatively smaller system. More importantly, part of slowly converging low-frequency error components on fine-grid will become high-frequency on coarse-grid and therefore can be further eliminated via local relaxation methods~\cite{Trottenberg2001}. By applying this process recursively, one can obtain a multilevel iterative method. Multigrid methods have been proved to possess uniform convergence with (nearly) optimal complexity for a large class of linear algebraic systems arising from the discretization of PDEs (see, e.g.,~\cite{Xu1992,Trottenberg2001,Vassilevski2008}).

Algebraic multigrid (AMG) was originally developed as a method for solving general matrix equations based on multigrid principles~\cite{Brandt1985,Ruge1985,Brandt1986,Ruge1987}. AMG constructs the coarsening process in a purely algebraic manner that requires no explicit knowledge of the geometric properties. More specifically, AMG determines inter-level transfer operators (restriction and prolongation) and coarse-level equations based only on the matrix entries; see the recent survey by Xu and Zikatanov~\cite{Xu2017}. AMG algorithms have gained increasing popularity among scientific and engineering computation due to successful applications to solve physical problems on unstructured grids~\cite{Cleary2000}.
An important ingredient of AMG algorithms is a strategy for constructing inter-level operators. When designing the prolongation operator in an AMG algorithm, it is desirable to be able to know its convergence quality a priori. To measure the quality of the coarse-grid in AMG, Falgout and Vassilevski~\cite{Falgout2004} studied the min-max property of the following measure:
\begin{equation}\label{g_meas}
\mu_{X}(Q,\mathbf{e}):=\frac{\big(X(I-Q)\mathbf{e},(I-Q)\mathbf{e}\big)}{(A\mathbf{e},\mathbf{e})}, \quad \forall\, \mathbf{e}\in\mathbb{R}^{n}\backslash\{0\},
\end{equation}
where both $A\in\mathbb{R}^{n\times n}$ and $X\in\mathbb{R}^{n\times n}$ are symmetric positive definite (SPD) and $Q=PR\in\mathbb{R}^{n\times n}$ (here $P\in\mathbb{R}^{n\times n_{c}}$, $R\in\mathbb{R}^{n_{c}\times n}$, and $RP=I_{n_{c}}$). 

Throughout this paper, a prolongation operator $P_{\star}\in\mathbb{R}^{n\times n_{c}}$ is referred to as an \emph{ideal} interpolation if
\begin{displaymath}
P_{\star}\in\mathop{\arg\min}_{P}\Big\{\max_{\mathbf{e}\neq 0}\mu_{X}(PR,\mathbf{e})\Big\}.
\end{displaymath}
It was argued by Falgout and Vassilevski~\cite[Theorem 3.1]{Falgout2004} that $P_{\star}$ must satisfy
\begin{equation}\label{key_rela1}
P_{\star}^{T}AS=0,
\end{equation}
where $S\in\mathbb{R}^{n\times n_{s}}$ $(n_{s}=n-n_{c})$ is of full column rank and $RS=0$. On the basis of~\eqref{key_rela1}, they derived an explicit expression for $P_{\star}$, i.e.,
\begin{equation}\label{old_exp2}
P_{\star}=\big(I-S(S^{T}AS)^{-1}S^{T}A\big)R^{T},
\end{equation}
which provides foundation for relating and comparing their theory to existing methods such as AMGe~\cite{Brezina2001,Jones2001}, spectral AMGe~\cite{Chartier2003}, and smoothed aggregation AMG~\cite{Vanek1996,Vanek2001,Brezina2012}. 

Unfortunately, the ideal interpolation $P_{\star}$ may not satisfy~\eqref{key_rela1} and~\eqref{old_exp2}; see the counter-example in Example~\ref{count_exam}. In fact, \eqref{key_rela1} is only sufficient to ensure that $P_{\star}$ is an ideal interpolation in general. Motivated by this observation, we revisit the min-max property of the measure~\eqref{g_meas} and obtain new characterizations of the ideal interpolation. The main result of this paper is that the following set relations (see Theorem~\ref{main_theo}) hold:
\begin{equation}\label{set_rela}
\mathbb{P}_{0}\subseteq\mathbb{P}_{2}=\mathbb{P}_{\star}\subseteq\mathbb{P}_{1},
\end{equation}
where
\begin{align}
\mathbb{P}_{\star}&:=\Big\{P:\max_{\mathbf{e}\neq 0}\mu_{X}(PR,\mathbf{e})=\mu_{X}^{\star}\Big\},\\
\mathbb{P}_{0}&:=\Big\{P:P^{T}AS=0\Big\},\\
\mathbb{P}_{1}&:=\Big\{P: \Null\big(P^{T}AS\big)\cap\big\{\mathbf{v}\in\mathbb{R}^{n_{s}}\backslash\{0\}:S^{T}AS\mathbf{v}=\lambda_{\min}(A_{X})S^{T}XS\mathbf{v}\big\}\neq\varnothing\Big\},\label{eqv_P1}
\\
\mathbb{P}_{2}&:=\Big\{P:\Null\big(P^{T}AS\big)\cap\big\{\mathbf{v}\in\mathbb{R}^{n_{s}}\backslash\{0\}:S^{T}BS\mathbf{v}=\lambda_{\min}(B_{X})S^{T}XS\mathbf{v}\big\}\neq\varnothing\Big\}.\label{eqv_P2}
\end{align}
Here, $\mu_{X}^{\star}$, $A_{X}$, $B_{X}$, and $B$ will be specified later in~\eqref{opt_value}, \eqref{def_Ax}, \eqref{def_Bx}, and~\eqref{def_B}, respectively. The relation~\eqref{set_rela} suggests that one has more room than $\mathbb{P}_{0}$ to construct an ideal interpolation. Another interesting result is that the following expression for the ideal interpolation in $\mathbb{P}_{0}$ (see Theorem~\ref{main_theo1}) holds:
\begin{equation}\label{exp_P0}
P_{\star}=A^{-1}R^{T}(RA^{-1}R^{T})^{-1}.
\end{equation}
By comparing~\eqref{exp_P0} with~\eqref{old_exp2}, we see that the new expression~\eqref{exp_P0} does not involve the auxiliary operator $S$.

The rest of this paper is organized as follows. In Section~\ref{sec:Pre}, we first briefly review the two-grid (TG) method and existing results on the ideal interpolation, and then give an example to illustrate that the ideal interpolation $P_{\star}$ may not satisfy~\eqref{key_rela1} and~\eqref{old_exp2}. In Section~\ref{sec:Cha}, we establish new characterizations of the ideal interpolation. In Section~\ref{sec:Exp}, we present a new expression for the ideal interpolation in $\mathbb{P}_{0}$, which does not involve the operator $S$. Finally, some conclusions are given in Section~\ref{sec:Con}.

\section{Preliminaries}
\label{sec:Pre}
\setcounter{equation}{0}

We first introduce some basic notation. The identity matrix of order $n$ is denoted by $I_{n}$ (or $I$ when its size is clear in the context). The range and the null space of a matrix are denoted by $\Range(\cdot)$ and $\Null(\cdot)$, respectively. The largest and smallest eigenvalues of a matrix are denoted by $\lambda_{\max}(\cdot)$ and $\lambda_{\min}(\cdot)$, respectively. The Euclidean inner product ($L^{2}$-inner product) and its associated norm are denoted by $(\cdot,\cdot)$ and $\|\cdot\|:=(\cdot,\cdot)^{1/2}$, respectively. For an SPD matrix $A$, the $A$-inner product and the $A$-norm (also called the \emph{energy norm}) are defined by $(\cdot,\cdot)_{A}:=(A\cdot,\cdot)$ and $\|\cdot\|_{A}:=(\cdot,\cdot)_{A}^{1/2}$, respectively.

\subsection{Two-grid method} 

Consider solving the linear system 
\begin{equation}\label{linear}
A\mathbf{u}=\mathbf{f},
\end{equation}
where $A\in\mathbb{R}^{n\times n}$ is SPD, $\mathbf{u}\in\mathbb{R}^{n}$, and $\mathbf{f}\in\mathbb{R}^{n}$. Given a nonsingular matrix $M\in\mathbb{R}^{n\times n}$ and an initial guess $\mathbf{u}_{0}\in\mathbb{R}^{n}$, we perform the following iteration:
\begin{equation}\label{iter}
\mathbf{u}_{k+1}=\mathbf{u}_{k}+M^{-1}(\mathbf{f}-A\mathbf{u}_{k}), \quad k=0,1,\ldots,
\end{equation}
where $M$ is called a \emph{smoother} and $\mathbf{f}-A\mathbf{u}_{k}$ is the \emph{residual} at the $k$-th iteration. Let $\mathbf{e}_{k}=\mathbf{u}-\mathbf{u}_{k}$. We then have
\begin{displaymath}
\mathbf{e}_{k+1}=(I-M^{-1}A)\mathbf{e}_{k}.
\end{displaymath}
A sufficient and necessary condition for~\eqref{iter} to be $A$-convergent (i.e., $\|I-M^{-1}A\|_{A}<1$) is that $M+M^{T}-A$ is SPD, which can be easily seen from the identity
\begin{displaymath}
\big\|(I-M^{-1}A)\mathbf{e}\big\|_{A}^{2}=(A\mathbf{e},\mathbf{e})-\big((M+M^{T}-A)M^{-1}A\mathbf{e},M^{-1}A\mathbf{e}\big), \quad \forall\, \mathbf{e}\in\mathbb{R}^{n}.
\end{displaymath}

Let $P: \mathbb{R}^{n_{c}}\mapsto\mathbb{R}^{n}$ be a \emph{prolongation} (or \emph{interpolation}) operator, where $\mathbb{R}^{n_{c}}$ is a lower-dimensional (coarse) vector space of size $n_{c}$. The operator $A_{c}=P^{T}AP\in\mathbb{R}^{n_{c}\times n_{c}}$ is the so-called \emph{Galerkin coarse-grid operator}. For an initial guess $\mathbf{u}$, the standard (symmetrized) TG method (see, e.g.,~\cite[Algorithm 1]{Hu2016}) for solving~\eqref{linear} can be described as follows: 
\begin{align*}
\text{\emph{Step 1}. } & \text{Presmoothing}: \mathbf{u}\leftarrow \mathbf{u}+M^{-1}(\mathbf{f}-A\mathbf{u});&\\
\text{\emph{Step 2}. } & \text{Restriction}: \mathbf{r}_{c}\leftarrow P^{T}(\mathbf{f}-A\mathbf{u});&\\
\text{\emph{Step 3}. } & \text{Coarse-grid correction}: \mathbf{e}_{c}\leftarrow A_{c}^{-1}\mathbf{r}_{c};&\\
\text{\emph{Step 4}. } & \text{Prolongation}: \mathbf{u}\leftarrow \mathbf{u}+P\mathbf{e}_{c};&\\
\text{\emph{Step 5}. } & \text{Postsmoothing}: \mathbf{u}\leftarrow \mathbf{u}+M^{-T}(\mathbf{f}-A\mathbf{u}).&
\end{align*}
It is easy to see that the iteration matrix $E_\text{TG}$ of the above TG method is
\begin{equation}\label{iter_mat}
E_\text{TG}=(I-M^{-T}A)(I-PA_{c}^{-1}P^{T}A)(I-M^{-1}A).
\end{equation}
For more theories about the TG method, we refer to~\cite{Falgout2005,Notay2007,Notay2015} and the references therein. 
By applying the TG method recursively, one can obtain a multilevel method for solving~\eqref{linear}.

As is well-known, the aim of AMG methods is to balance the interplay between the smoother $M$ and the coarse-space $\Range(P)$. When a smoother $M$ is selected, the main task of an AMG algorithm is to construct a ``good'' prolongation $P$. Roughly speaking, $P$ should be constructed so that ``\emph{algebraically smooth error}'' can be effectively eliminated in correction steps and the coarse-grid equations (involving $A_{c}$) are amenable to solution~\cite{Brezina2001}. Here, ``algebraically smooth error'' refers to the error components that are not being effectively damped by the relaxation process~\eqref{iter}. 

\subsection{Quality measures and the ideal interpolation}

Let $R: \mathbb{R}^{n}\mapsto\mathbb{R}^{n_{c}}$ be an operator for which $RP=I_{n_{c}}$ and let $Q=PR\in\mathbb{R}^{n\times n}$. It is easy to see that $Q$ is a projection (i.e., $Q^{2}=Q$) onto $\Range(P)$. Let $S: \mathbb{R}^{n_{s}}\mapsto\mathbb{R}^{n}$ $(n_{s}=n-n_{c})$ be a full column rank operator satisfying $RS=0$. Clearly, $S$ and $R^{T}$ form an $L^{2}$-orthogonal decomposition of $\mathbb{R}^{n}$. That is, for any $\mathbf{e}\in\mathbb{R}^{n}$, it can be written as $\mathbf{e}=S\mathbf{e}_{s}+R^{T}\mathbf{e}_{c}$ for some $\mathbf{e}_{s}\in\mathbb{R}^{n_{s}}$ and $\mathbf{e}_{c}\in\mathbb{R}^{n_{c}}$. 

In the classical AMG setting, the set of coarse-grid variables is a subset of fine-grid variables. Typically, the operators $R$, $S$, and $P$ have the following forms:
\begin{displaymath}
R=\begin{pmatrix}
0 & I_{n_{c}}
\end{pmatrix}, \quad S=\begin{pmatrix}
I_{n_{s}}\\
0
\end{pmatrix}, \quad 
P=\begin{pmatrix}
W\\
I_{n_{c}}
\end{pmatrix},
\end{displaymath}
where $W\in\mathbb{R}^{n_{s}\times n_{c}}$ denotes the interpolation weights for fine-grid variables.

Since $Q=PR$ and $RP=I_{n_{c}}$, for any $\mathbf{e}\in\Range(P)$, we have $(I-Q)\mathbf{e}=0$. Thus, $I-Q$ can be used to measure the defect of $P$. Define
\begin{equation}\label{mu_M}
\mu_{\widetilde{M}}(Q,\mathbf{e}):=\frac{\big(\widetilde{M}(I-Q)\mathbf{e},(I-Q)\mathbf{e}\big)}{(A\mathbf{e},\mathbf{e})},
\end{equation}
where
\begin{displaymath}
\widetilde{M}:=M^{T}(M+M^{T}-A)^{-1}M.
\end{displaymath}
Let
\begin{equation}\label{def_K}
K=\sup_{\mathbf{e}\neq 0}\mu_{\widetilde{M}}(Q,\mathbf{e}),
\end{equation}
and let $E_\text{TG}$ be given by~\eqref{iter_mat}. Falgout and Vassilevski~\cite[Theorem 2.2]{Falgout2004} showed that $K\geq 1$ and
\begin{equation}\label{conver}
\|E_\text{TG}\|_{A}=\big\|(I-M^{-T}A)(I-PA_{c}^{-1}P^{T}A)\big\|_{A}^{2}\leq 1-\frac{1}{K}.
\end{equation}
This shows that, if the measure $\mu_{\widetilde{M}}$ is bounded above by a constant, then the TG method converges uniformly.

\begin{remark}\rm
A simpler measure
\begin{equation}\label{mu_s}
\mu_{M_\text{s}}(Q,\mathbf{e}):=\frac{\big(M_\text{s}(I-Q)\mathbf{e},(I-Q)\mathbf{e}\big)}{(A\mathbf{e},\mathbf{e})}
\end{equation}
was given in~\cite[Eq. (2.11)]{Falgout2004}, where 
\begin{displaymath}
M_\text{s}:=\frac{1}{2}(M+M^{T})
\end{displaymath}
is the symmetric part of $M$. Assume that $M+M^{T}-A$ is SPD. It was proved by Falgout and Vassilevski~\cite[Lemma 2.3]{Falgout2004} that
\begin{equation}\label{mu_MH}
\mu_{\widetilde{M}}(Q,\mathbf{e})\leq\frac{\Delta^{2}}{2-\omega}\mu_{M_\text{s}}(Q,\mathbf{e}),
\end{equation}
where $0<\omega:=\lambda_{\max}\big(M_\text{s}^{-1}A\big)<2$ and $\Delta\geq1$ measures the deviation of $M$ from its symmetric part $M_\text{s}$ in the sense that
\begin{displaymath}
(M\mathbf{v},\mathbf{w})\leq\Delta(M_\text{s}\mathbf{v},\mathbf{v})^{1/2}(M_\text{s}\mathbf{w},\mathbf{w})^{1/2}.
\end{displaymath}
The relation~\eqref{mu_MH} suggests that the uniform upper bound for $\mu_{\widetilde{M}}$ can be acquired by bounding $\mu_{M_\text{s}}$ uniformly.
\end{remark}

To analyze the min-max properties of $\mu_{\widetilde{M}}$ and $\mu_{M_\text{s}}$, the general measure~\eqref{g_meas} was considered in~\cite{Falgout2004}. Assume that the measure~\eqref{g_meas} is bounded uniformly for all $\mathbf{e}\in\mathbb{R}^{n}\backslash\{0\}$ (without loss of generality, we assume that $\|\mathbf{e}\|=1$). If $\mathbf{e}$ is an eigenvector of $A$ corresponding to a small eigenvalue, then the denominator is small and thus the numerator must be small as well. Hence, $Q$ can accurately interpolate the eigenvectors corresponding to the small eigenvalues of $A$. On the other hand, if $\mathbf{e}$ is an eigenvector of $A$ corresponding to a large eigenvalue, then the denominator is large, which implies that the numerator may be large. Hence, $Q$ may not interpolate the eigenvectors corresponding to the large eigenvalues of $A$ accurately~\cite{Brezina2001}.

Actually, there are many choices to select the SPD matrix $X$ in~\eqref{g_meas}. For example, $X$ can be selected so that it is spectrally equivalent to $\widetilde{M}$, namely,
\begin{equation}\label{spe_eqv}
c_{1}\mathbf{e}^{T}X\mathbf{e}\leq\mathbf{e}^{T}\widetilde{M}\mathbf{e}\leq c_{2}\mathbf{e}^{T}X\mathbf{e}, \quad \forall\,\mathbf{e}\in\mathbb{R}^{n},
\end{equation}
where $c_{1}$ and $c_{2}$ are two generic positive constants. We next give an interpretation for this choice. Let $\eta=\sup_{\mathbf{e}\neq 0}\mu_{X}(Q,\mathbf{e})$ and $\mathbf{e}_{c}=R\mathbf{e}$. We then have
\begin{equation}\label{WAP}
\|\mathbf{e}-P\mathbf{e}_{c}\|_{X}^{2}\leq\eta\|\mathbf{e}\|_{A}^{2}, \quad \forall\,\mathbf{e}\in\mathbb{R}^{n}.
\end{equation}
If the relation~\eqref{spe_eqv} holds, we have
\begin{displaymath}
\|\mathbf{e}-P\mathbf{e}_{c}\|_{\widetilde{M}}^{2}\leq c_{2}\|\mathbf{e}-P\mathbf{e}_{c}\|_{X}^{2}\leq c_{2}\eta\|\mathbf{e}\|_{A}^{2}, \quad \forall\,\mathbf{e}\in\mathbb{R}^{n},
\end{displaymath}
which yields $K\leq c_{2}\eta$ ($K$ is given by~\eqref{def_K}). Using~\eqref{conver}, we immediately obtain
\begin{displaymath}
\|E_\text{TG}\|_{A}\leq 1-\frac{1}{c_{2}\eta},
\end{displaymath}
which implies the uniform convergence of the TG method. The \emph{weak approximation property} of the coarse-space can be stated as: for any $\mathbf{e}\in\mathbb{R}^{n}$, there is a coarse vector $\mathbf{e}_{c}\in\mathbb{R}^{n_{c}}$ such that~\eqref{WAP} holds, provided that $X$ is spectrally equivalent to $\widetilde{M}$. It is well-known that the weak approximation property is a sufficient and necessary condition for the uniform convergence of the TG method (see, e.g.,~\cite[Chapter 5, Section 3]{Vassilevski2010}).

The following lemma presents the min-max property of the measure~\eqref{g_meas}, which gives a necessary condition and an explicit expression of the ideal interpolation~\cite[Theorem 3.1 and Corollary 3.2]{Falgout2004}.

\begin{lemma}	
Let $A\in\mathbb{R}^{n\times n}$, $X\in\mathbb{R}^{n\times n}$,
$P\in\mathbb{R}^{n\times n_{c}}$, $R\in\mathbb{R}^{n_{c}\times n}$, and $S\in\mathbb{R}^{n\times n_{s}}$ ($n_{s}=n-n_{c}$). Assume that both $A$ and $X$ are SPD, $RP=I_{n_{c}}$, $RS=0$, and $S$ is of full column rank. Define
\begin{equation}\label{def_opt_value}
\mu_{X}^{\star}:=\min_{P}\max_{\mathbf{e}\neq 0}\mu_{X}(PR,\mathbf{e}),
\end{equation}
where $\mu_{X}(PR,\mathbf{e})$ is defined by~\eqref{g_meas}. Then
\begin{equation}\label{opt_value}
\mu_{X}^{\star}=\frac{1}{\lambda_{\min}\big((S^{T}XS)^{-1}S^{T}AS\big)},
\end{equation}
and the minimizer $P_{\star}$ must satisfy
\begin{equation}\label{key_rela2}
P_{\star}^{T}AS=0.
\end{equation}
Moreover, $P_{\star}$ has the explicit expression
\begin{equation}\label{old_exp}
P_{\star}=\big(I-S(S^{T}AS)^{-1}S^{T}A\big)R^{T}.
\end{equation}
\end{lemma}

\begin{remark}\rm
If we set $X$ as $\widetilde{M}$ and $M_\text{s}$, then~\eqref{def_opt_value} reduce to
\begin{displaymath}
\mu_{\widetilde{M}}^{\star}:=\min_{P}\max_{\mathbf{e}\neq 0}\mu_{\widetilde{M}}(PR,\mathbf{e})
\quad \text{and} \quad
\mu_{M_\text{s}}^{\star}:=\min_{P}\max_{\mathbf{e}\neq 0}\mu_{M_\text{s}}(PR,\mathbf{e}),
\end{displaymath}
respectively. The quantities $\mu_{\widetilde{M}}^{\star}$ and $\mu_{M_\text{s}}^{\star}$ measure the ability of the coarse-grid to represent ``algebraically smooth error''. Empirical evidence so far indicates that $\mu_{\widetilde{M}}^{\star}$ and $\mu_{M_\text{s}}^{\star}$ are useful measures in practice~\cite{Falgout2004}.
\end{remark}

\subsection{Geometric illustration of the measure $\mu_{\widetilde{M}}$}

Let $E_\text{TG}$ be given by~\eqref{iter_mat} and define
\begin{equation}\label{pi_M}
\varPi_{\widetilde{M}}:=P(P^{T}\widetilde{M}P)^{-1}P^{T}\widetilde{M}.
\end{equation}
Falgout et al.~\cite[Theorem 4.3]{Falgout2005} proved that
\begin{equation}\label{E_TG}
\|E_\text{TG}\|_{A}=1-\frac{1}{K_\text{TG}},
\end{equation}
where
\begin{equation}\label{K_TG}
K_\text{TG}=\sup_{\mathbf{e}\neq 0} \, \frac{ \mathbf{e}^{T} (I-\varPi_{\widetilde{M}})^{T} \widetilde{M}(I-\varPi_{\widetilde{M}})\mathbf{e}}{\mathbf{e}^{T}A\mathbf{e}}.
\end{equation}
As discussed in~\cite[Remark 4.1]{Falgout2005}, for any $\mathbf{e}\in\mathbb{R}^{n}$, it holds that
\begin{equation}\label{proj_rela}
\big\|(I-\varPi_{\widetilde{M}})\mathbf{e}\big\|_{\widetilde{M}}^{2}\leq\big\|(I-Q)\mathbf{e}\big\|_{\widetilde{M}}^{2},
\end{equation}
where $Q=PR$ and $R: \mathbb{R}^{n}\mapsto\mathbb{R}^{n_{c}}$ is an operator satisfying $RP=I_{n_{c}}$. By combining~\eqref{K_TG} and~\eqref{proj_rela}, we obtain that
\begin{equation}\label{K_TG,K}
K_\text{TG}\leq\sup_{\mathbf{e}\neq 0}\frac{\|(I-Q)\mathbf{e}\|_{\widetilde{M}}^{2}}{\|\mathbf{e}\|_{A}^{2}}=\sup_{\mathbf{e}\neq 0}\mu_{\widetilde{M}}(Q,\mathbf{e})=K.
\end{equation}
Hence, the estimate~\eqref{conver} follows immediately from~\eqref{E_TG} and~\eqref{K_TG,K}.

If an operator $P_\text{opt}\in\mathbb{R}^{n\times n_{c}}$ directly minimizes the TG convergence rate $\|E_\text{TG}\|_{A}$, then $P_\text{opt}$ is called the \emph{optimal} interpolation operator. In view of~\eqref{E_TG}, we can obtain a lower bound for $K_\text{TG}$, that is,
\begin{displaymath}
K_\text{TG}\geq\frac{1}{1-\|E_\text{TG}(P_\text{opt})\|_{A}}.
\end{displaymath}
Clearly, $P_\text{opt}$ is also the interpolation operator that minimizes $K_\text{TG}$. The optimal interpolation $P_\text{opt}$ can provide guidance in the design of practical AMG methods. However, $P_\text{opt}$ itself is expensive to compute due to its columns consist of eigenvectors corresponding to small eigenvalues. Explicit form of $P_\text{opt}$ (or the optimal coarse-space) and the precise value $\|E_\text{TG}(P_\text{opt})\|_{A}$ can be found, e.g., in~\cite{Xu2017,Brannick2017}. Recently, some interesting relationships between the optimal and ideal interpolations have been discussed by Brannick et al.~\cite{Brannick2017}.

It is easy to see that the relation~\eqref{proj_rela} is equivalent to
\begin{equation}\label{trans_rela}
\Big\|\big(I-\widetilde{M}^{1/2}\varPi_{\widetilde{M}}\widetilde{M}^{-1/2}\big)\widetilde{M}^{1/2}\mathbf{e}\Big\|\leq\Big\|\big(I-\widetilde{M}^{1/2}PR\widetilde{M}^{-1/2}\big)\widetilde{M}^{1/2}\mathbf{e}\Big\|.
\end{equation} 
According to the definition~\eqref{pi_M}, we have that $I-\widetilde{M}^{1/2}\varPi_{\widetilde{M}}\widetilde{M}^{-1/2}$ is an $L^{2}$-orthogonal projection \emph{along} (or \emph{parallel} to) $\Range(\widetilde{M}^{1/2}P)$ onto $\Null(P^{T}\widetilde{M}^{1/2})$. Similarly, $I-\widetilde{M}^{1/2}PR\widetilde{M}^{-1/2}$ is an $L^{2}$-oblique projection along $\Range(\widetilde{M}^{1/2}P)$ onto $\Null(R\widetilde{M}^{-1/2})$. In two-dimensional case, a geometric illustration of~\eqref{trans_rela} is shown in Figure~\ref{fig:2D_proj}.

\begin{figure}[htbp]
\centering
\includegraphics[width=0.58\textwidth,height=0.32\textheight]{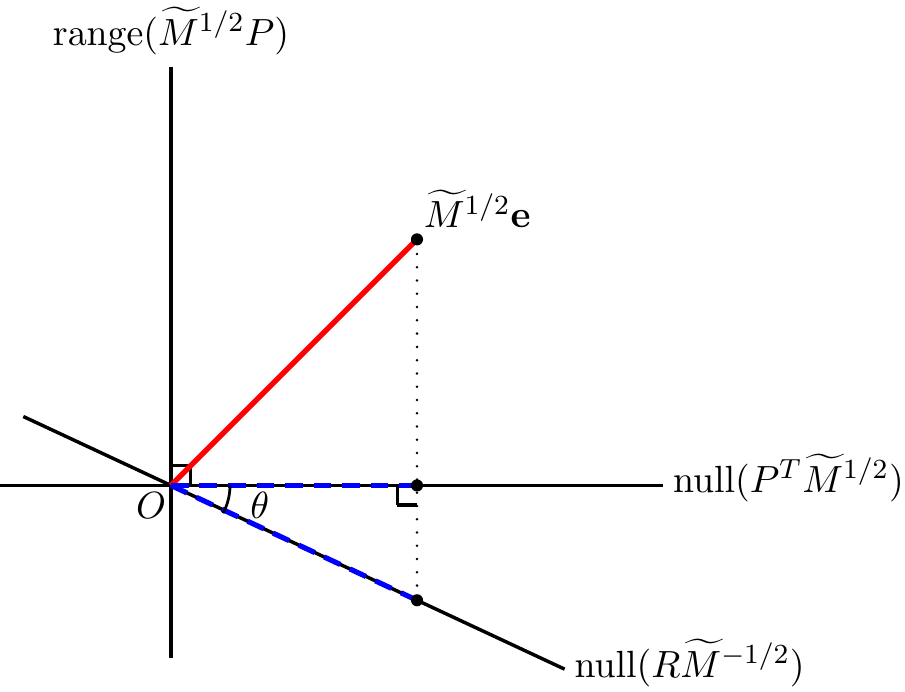} 
\caption{\small Two-dimensional illustration of \eqref{trans_rela}.}
\label{fig:2D_proj}
\end{figure}

From Figure~\ref{fig:2D_proj}, we observe that
\begin{displaymath}
\Big\|\big(I-\widetilde{M}^{1/2}\varPi_{\widetilde{M}}\widetilde{M}^{-1/2}\big)\widetilde{M}^{1/2}\mathbf{e}\Big\|=\Big\|\big(I-\widetilde{M}^{1/2}PR\widetilde{M}^{-1/2}\big)\widetilde{M}^{1/2}\mathbf{e}\Big\|\cdot\cos\theta,
\end{displaymath}
where $\theta$ $(0\leq\theta<\pi/2)$ denotes the angle between the spaces $\Null(P^{T}\widetilde{M}^{1/2})$ and $\Null(R\widetilde{M}^{-1/2})$. Hence,
\begin{displaymath}
K=\frac{1}{\cos^{2}\theta}K_\text{TG}.
\end{displaymath}
Obviously, $K$ approaches $K_\text{TG}$ as $\theta$ tends to zero. In other words, if the angle $\theta$ is small, then $\mu_{\widetilde{M}}$ can measure the quality of the coarse-grid effectively.

\begin{remark}\rm
Define
\begin{displaymath}
P_{\sharp}:=\widetilde{M}^{-1}R^{T}\big(R\widetilde{M}^{-1}R^{T}\big)^{-1}.
\end{displaymath}
Based on the relation $RP=I_{n_{c}}$, we can derive that $\theta=0$ (or $\Null(P^{T}\widetilde{M}^{1/2})=\Null(R\widetilde{M}^{-1/2})$) if and only if $P=P_{\sharp}$. In this case, $K=K_\text{TG}$ and hence $\|E_\text{TG}\|_{A}=1-1/K$.
\end{remark}

\subsection{An illustrative example}

We are now in a position to illustrate that the ideal interpolation $P_{\star}$ may not satisfy~\eqref{key_rela2} and~\eqref{old_exp}.

\begin{example}\label{count_exam}\rm
Let
\begin{displaymath}
A=\begin{pmatrix}
2 & -1 & 1\\
-1 & 2 & -1\\
1 & -1 & 2
\end{pmatrix}, \quad
X=\diag(A)=\begin{pmatrix}
2 & 0 & 0\\
0 & 2 & 0\\
0 & 0 & 2
\end{pmatrix}, \quad
S=\begin{pmatrix}
1 & 0\\
0 & 1\\
0 & 0
\end{pmatrix}, 
\quad 
\text{and}
\quad 
P=R^{T}=\begin{pmatrix}
0\\
0\\
1
\end{pmatrix}.
\end{displaymath}
Straightforward calculations yield
\begin{displaymath}
S^{T}AS=\begin{pmatrix}
2 & -1\\
-1 & 2
\end{pmatrix} \quad \text{and} \quad S^{T}XS=\begin{pmatrix}
2 & 0\\
0 & 2
\end{pmatrix}.
\end{displaymath}
We remark that, although the ideal interpolation may not satisfy~\eqref{key_rela2} and~\eqref{old_exp}, the value $\mu_{X}^{\star}$ given by~\eqref{opt_value} is correct. By~\eqref{opt_value} and~\eqref{old_exp}, we have
\begin{displaymath}
\mu_{X}^{\star}=2 \quad \text{and} \quad P_{\star}=\begin{pmatrix}
-\frac{1}{3}\\
\frac{1}{3}\\
1
\end{pmatrix}.
\end{displaymath}
On the other hand, it is easy to see that
\begin{displaymath}
P^{T}AS=\begin{pmatrix}
1 & -1
\end{pmatrix}.
\end{displaymath}
Direct computation yields
\begin{displaymath}
\max_{\mathbf{e}\neq 0}\mu_{X}(PR,\mathbf{e})=2=\mu_{X}^{\star}.
\end{displaymath}
Thus, $P$ is an ideal interpolation (however, $P^{T} A S \neq 0$ and $P\neq P_{\star}$). Indeed, in this example, both $P$ and $P_{\star}$ are ideal interpolations.
\end{example}

To test the numerical performances of $P$ and $P_{\star}$, we perform a simple experiment. Set 
\begin{displaymath}
\mathbf{f}=\begin{pmatrix}
1\\
1\\
1
\end{pmatrix}, \quad \mathbf{u}_{0}=\begin{pmatrix}
0\\
0\\
0
\end{pmatrix}, \quad \text{and} \quad 
M=\frac{1}{0.8}\diag(A)=\begin{pmatrix}
2.5 & 0 & 0\\
0 & 2.5 & 0\\
0 & 0 & 2.5
\end{pmatrix}.
\end{displaymath}
Evidently, $M$ is a weighted Jacobi type smoother. We take $X=\frac{1}{2}(M+M^{T})$ as in~\eqref{mu_s}, which is a scalar matrix and hence $P$ in Example~\ref{count_exam} is still an ideal interpolation. We solve the linear system $A\mathbf{u}=\mathbf{f}$ (with the initial guess $\mathbf{u}_{0}$) by using the above TG method. Applying $P$ and $P_{\star}$ as the prolongation operators, respectively, we find that the required numbers of iterations are $15$ for both choices in order to make the residuals decrease by $6$ magnitudes.

In conclusion, Example~\ref{count_exam} demonstrates that the minimizer of~\eqref{def_opt_value} may not satisfy~\eqref{key_rela2}, and the ideal interpolation $P_{\star}$ given by~\eqref{old_exp} is not the unique minimizer of~\eqref{def_opt_value} in general. Motivated by this observation, we revisit the min-max property of the measure~\eqref{g_meas}. Some new characterizations of the ideal interpolation will be shown in the next section.

\section{Characterizations of the ideal interpolation}
\label{sec:Cha}
\setcounter{equation}{0}

In this section, we establish some new characterizations of the ideal interpolation, which can provide guidance for designing new AMG algorithms.
 
Some conditions are required for our analysis, which are summarized as follows:
\begin{displaymath}
(\mathbf{C}):
\begin{cases} A\in\mathbb{R}^{n\times n},\ X\in\mathbb{R}^{n\times n},\
P\in\mathbb{R}^{n\times n_{c}},\ R\in\mathbb{R}^{n_{c}\times n},\ 
S\in\mathbb{R}^{n\times n_{s}}\ (n_{s}=n-n_{c}), \ Q=PR,\\
RP=I_{n_{c}},\ RS=0,\ \emph{both $A$ and $X$ are SPD matrices},\ \emph{and $S$ is of full column rank}.
\end{cases}
\end{displaymath}
From the condition $(\mathbf{C})$, we can easily see that $S$ and $R^{T}$ form an $L^{2}$-orthogonal decomposition of $\mathbb{R}^{n}$. In addition, both $(S \ P)\in\mathbb{R}^{n\times n}$ and $(S \ R^{T})\in\mathbb{R}^{n\times n}$ are nonsingular. The following lemma gives the explicit expressions for $(S \ P)^{-1}$ and $(S \ R^{T})^{-1}$.

\begin{lemma}
Under the condition $(\mathbf{C})$, the matrices $(S \ P)^{-1}$ and $(S \ R^{T})^{-1}$ have the following expressions:
\begin{align}
\begin{pmatrix}
S & P
\end{pmatrix}^{-1}&=\begin{pmatrix}
(S^{T}AS)^{-1}S^{T}A(I-Q)\\
R
\end{pmatrix},\label{inv_SP}\\
\begin{pmatrix}
S & R^{T}
\end{pmatrix}^{-1}&=\begin{pmatrix}
(S^{T}AS)^{-1}S^{T}A\big(I-R^{T}(RR^{T})^{-1}R\big)\\
(RR^{T})^{-1}R
\end{pmatrix}.\label{inv_SR}
\end{align}
\end{lemma}

\begin{proof}
Due to $RS=0$ and $RP=I_{n_{c}}$, it follows that 
\begin{displaymath}
(I-Q)S=S \quad \text{and} \quad (I-Q)P=0.
\end{displaymath}
Direct computations yield
\begin{align*}
\begin{pmatrix}
(S^{T}AS)^{-1}S^{T}A(I-Q)\\
R
\end{pmatrix}\begin{pmatrix}
S & P
\end{pmatrix}&=I,\\
\begin{pmatrix}
(S^{T}AS)^{-1}S^{T}A(I-R^{T}(RR^{T})^{-1}R)\\
(RR^{T})^{-1}R
\end{pmatrix}\begin{pmatrix}
S & R^{T}
\end{pmatrix}&=I.
\end{align*}
Hence, the expressions~\eqref{inv_SP} and~\eqref{inv_SR} are verified.
\end{proof}

In view of~\eqref{inv_SP} and~\eqref{inv_SR}, we can derive the following lemma, which presents two equivalent forms of $I-S(S^{T}AS)^{-1}S^{T}A$ and a general expression for $P$.

\begin{lemma} 
Under the condition $(\mathbf{C})$, we have the following results:

{\rm (i)} $I-S(S^{T}AS)^{-1}S^{T}A$ has the following equivalent forms:
\begin{align}
I-S(S^{T}AS)^{-1}S^{T}A&=\big(I-S(S^{T}AS)^{-1}S^{T}A\big)Q,\label{equiv1}\\
&=\big(I-S(S^{T}AS)^{-1}S^{T}A\big)R^{T}(RR^{T})^{-1}R;\label{equiv2}
\end{align}

{\rm (ii)} $P$ can be expressed as
\begin{equation}\label{g_exp}
P=R^{T}(RR^{T})^{-1}+S(S^{T}AS)^{-1}S^{T}AY
\end{equation}
for some $Y\in\mathbb{R}^{n\times n_{c}}$.
\end{lemma}

\begin{proof}
(i) By~\eqref{inv_SP}, we have
\begin{displaymath}
\begin{pmatrix}
S & P
\end{pmatrix}\begin{pmatrix}
(S^{T}AS)^{-1}S^{T}A(I-Q)\\
R
\end{pmatrix}=I,
\end{displaymath}
which implies~\eqref{equiv1}. Similarly, we can verify~\eqref{equiv2} based on the equality~\eqref{inv_SR}. 

(ii) Combining~\eqref{equiv1} and~\eqref{equiv2}, we get
\begin{displaymath}
\big(I-S(S^{T}AS)^{-1}S^{T}A\big)PR=\big(I-S(S^{T}AS)^{-1}S^{T}A\big)R^{T}(RR^{T})^{-1}R.
\end{displaymath}
Using $RP=I_{n_{c}}$, we obtain
\begin{displaymath}
\big(I-S(S^{T}AS)^{-1}S^{T}A\big)P=\big(I-S(S^{T}AS)^{-1}S^{T}A\big)R^{T}(RR^{T})^{-1},
\end{displaymath}
which yields
\begin{displaymath}
\big(I-S(S^{T}AS)^{-1}S^{T}A\big)\big(P-R^{T}(RR^{T})^{-1}\big)=0.
\end{displaymath}
Due to
\begin{displaymath}
\Null\big(I-S(S^{T}AS)^{-1}S^{T}A\big)=\Range\big(S(S^{T}AS)^{-1}S^{T}A\big),
\end{displaymath}
it follows that
\begin{displaymath}
P=R^{T}(RR^{T})^{-1}+S(S^{T}AS)^{-1}S^{T}AY
\end{displaymath}
for some $Y\in\mathbb{R}^{n\times n_{c}}$.
\end{proof}

Let $\mu_{X}(PR,\mathbf{e})$ and $\mu_{X}^{\star}$ be defined by~\eqref{g_meas} and~\eqref{def_opt_value}, respectively. We define the set of all ideal interpolations as follows:
\begin{equation}\label{ideal_set1}
\mathbb{P}_{\star}:=\Big\{P:\max_{\mathbf{e}\neq 0}\mu_{X}(PR,\mathbf{e})=\mu_{X}^{\star}\Big\}.
\end{equation}
It is not easy to acquire the properties of the ideal interpolation from the formal definition~\eqref{ideal_set1}. Alternatively, the following lemma presents an equivalent characterization of the set $\mathbb{P}_{\star}$, which gives a clearer interpretation of the ideal interpolation.

\begin{lemma}
Under the condition $(\mathbf{C})$, the set $\mathbb{P}_{\star}$ defined by~\eqref{ideal_set1} can be expressed as
\begin{equation}\label{ideal_set2}
\mathbb{P}_{\star}=\big\{P:\lambda_{\min}(B_{X})=\lambda_{\min}(A_{X})\big\},
\end{equation}
where
\begin{align}
A_{X}&:=(S^{T}XS)^{-1/2}S^{T}AS(S^{T}XS)^{-1/2},\label{def_Ax}\\
B_{X}&:=(S^{T}XS)^{-1/2}\big(S^{T}AS-S^{T}AP(P^{T}AP)^{-1}P^{T}AS\big)(S^{T}XS)^{-1/2}.\label{def_Bx}
\end{align}
\end{lemma}

\begin{proof}
Since $Q=PR$ is a projection and $P$ is of full column rank, we have
\begin{displaymath}
\Range(I-Q)=\Null(PR)=\Null(R)=\Range(S),
\end{displaymath}
where we have used the fact that $S$ and $R^{T}$ form an $L^{2}$-orthogonal decomposition of $\mathbb{R}^{n}$. Hence, for any $\mathbf{e}\in\mathbb{R}^{n}$, we have
\begin{displaymath}
\mathbf{e}-PR\mathbf{e}\in\Range(S).
\end{displaymath}
Let $\mathbf{e}_{c}=R\mathbf{e}\in\mathbb{R}^{n_{c}}$. We then have, for some $\mathbf{e}_{s}\in\mathbb{R}^{n_{s}}$,
\begin{displaymath}
\mathbf{e}=S\mathbf{e}_{s}+P\mathbf{e}_{c}.
\end{displaymath}

Note that $(I-Q)P=0$ and $(I-Q)S=S$.
According to~\eqref{g_meas} and~\eqref{def_opt_value}, we have
\begin{align*}
\mu_{X}^{\star}&=\min_{P}\max_{\mathbf{e}\neq 0}\frac{(X(I-Q)\mathbf{e},(I-Q)\mathbf{e})}{(A\mathbf{e},\mathbf{e})}\\
&=\min_{P}\max_{\mathbf{e}_{s}\neq 0}\frac{(S^{T}XS\mathbf{e}_{s},\mathbf{e}_{s})}{\min_{\mathbf{e}_{c}}\big\{(S^{T}AS\mathbf{e}_{s},\mathbf{e}_{s})+2(P^{T}AS\mathbf{e}_{s},\mathbf{e}_{c})+(P^{T}AP\mathbf{e}_{c},\mathbf{e}_{c})\big\}}\\
&=\min_{P}\max_{\mathbf{e}_{s}\neq 0}\frac{(S^{T}XS\mathbf{e}_{s},\mathbf{e}_{s})}{(S^{T}AS\mathbf{e}_{s},\mathbf{e}_{s})-\big((P^{T}AP)^{-1}P^{T}AS\mathbf{e}_{s},P^{T}AS\mathbf{e}_{s}\big)}.
\end{align*}
Due to $(S \ P)$ is nonsingular (see~\eqref{inv_SP}) and $A$ is SPD, it follows that 
\begin{displaymath}
\begin{pmatrix}
S^{T}\\
P^{T}
\end{pmatrix}A
\begin{pmatrix}
S & P
\end{pmatrix}=
\begin{pmatrix}
S^{T}AS & S^{T}AP\\
P^{T}AS & P^{T}AP
\end{pmatrix}
\end{displaymath}
is also SPD. Letting
\begin{displaymath}
U=\begin{pmatrix}
I_{n_{s}} & -S^{T}AP(P^{T}AP)^{-1}\\
0 & I_{n_{c}}
\end{pmatrix},
\end{displaymath}
we then have
\begin{displaymath}
U\begin{pmatrix}
S^{T}AS & S^{T}AP\\
P^{T}AS & P^{T}AP
\end{pmatrix}U^{T}=\begin{pmatrix}
S^{T}AS-S^{T}AP(P^{T}AP)^{-1}P^{T}AS & 0\\
0 & P^{T}AP
\end{pmatrix},
\end{displaymath}
which implies that both $S^{T}AS-S^{T}AP(P^{T}AP)^{-1}P^{T}AS$ and $P^{T}AP$ are SPD. Hence,
\begin{displaymath}
\frac{(S^{T}XS\mathbf{e}_{s},\mathbf{e}_{s})}{(S^{T}AS\mathbf{e}_{s},\mathbf{e}_{s})-\big((P^{T}AP)^{-1}P^{T}AS\mathbf{e}_{s},P^{T}AS\mathbf{e}_{s}\big)}\geq\frac{(S^{T}XS\mathbf{e}_{s},\mathbf{e}_{s})}{(S^{T}AS\mathbf{e}_{s},\mathbf{e}_{s})}, \quad \forall \, \mathbf{e}_{s}\in \mathbb{R}^{n_{s}}\backslash\{0\},
\end{displaymath}
which yields
\begin{displaymath}
\max_{\mathbf{e}_{s}\neq 0}\frac{(S^{T}XS\mathbf{e}_{s},\mathbf{e}_{s})}{(S^{T}AS\mathbf{e}_{s},\mathbf{e}_{s})-\big(S^{T}AP(P^{T}AP)^{-1}P^{T}AS\mathbf{e}_{s},\mathbf{e}_{s}\big)}\geq\max_{\mathbf{e}_{s}\neq 0}\frac{(S^{T}XS\mathbf{e}_{s},\mathbf{e}_{s})}{(S^{T}AS\mathbf{e}_{s},\mathbf{e}_{s})}.
\end{displaymath}
That is,
\begin{displaymath}
\lambda_{\max}\big(\big(S^{T}AS-S^{T}AP(P^{T}AP)^{-1}P^{T}AS\big)^{-1}S^{T}XS\big)\geq\lambda_{\max}\big((S^{T}AS)^{-1}S^{T}XS\big).
\end{displaymath}
Therefore,
\begin{equation}\label{inequ1}
\max_{\mathbf{e}\neq 0}\mu_{X}(PR,\mathbf{e})=\frac{1}{\lambda_{\min}(B_{X})}\geq\frac{1}{\lambda_{\min}(A_{X})}=\mu_{X}^{\star}.
\end{equation}
In view of~\eqref{ideal_set1} and~\eqref{inequ1}, we immediately get the equivalent expression~\eqref{ideal_set2}. 
\end{proof}

The following corollary presents a variant of~\eqref{ideal_set2}.

\begin{corollary}
The set $\mathbb{P}_{\star}$ given by~\eqref{ideal_set2} can be expressed as
\begin{displaymath}
\mathbb{P}_{\star}=\Big\{P:\sigma_{\min}\big(A^{1/2}\big(I-P(P^{T}AP)^{-1}P^{T}A\big)S(S^{T}XS)^{-1/2}\big)=\sigma_{\min}\big(A^{1/2}S(S^{T}XS)^{-1/2}\big)\Big\},
\end{displaymath}
where $\sigma_{\min}(\cdot)$ denotes the smallest singular value of a matrix.
\end{corollary}

\begin{proof}
From~\eqref{def_Ax}, we have
\begin{displaymath}
A_{X}=\big((S^{T}XS)^{-1/2}S^{T}A^{1/2}\big)\big(A^{1/2}S(S^{T}XS)^{-1/2}\big),
\end{displaymath}
which implies
\begin{displaymath}
\lambda_{\min}(A_{X})=\sigma_{\min}^{2}\big(A^{1/2}S(S^{T}XS)^{-1/2}\big).
\end{displaymath}
From~\eqref{def_Bx}, we have
\begin{displaymath}
B_{X}=(S^{T}XS)^{-1/2}S^{T}A^{1/2}\big(I-A^{1/2}P(P^{T}AP)^{-1}P^{T}A^{1/2}\big)A^{1/2}S(S^{T}XS)^{-1/2}.
\end{displaymath}
Note that $I-A^{1/2}P(P^{T}AP)^{-1}P^{T}A^{1/2}$ is an $L^{2}$-orthogonal projection. We then have
\begin{align*}
\lambda_{\min}(B_{X})&=\sigma_{\min}^{2}\big(\big(I-A^{1/2}P(P^{T}AP)^{-1}P^{T}A^{1/2}\big)A^{1/2}S(S^{T}XS)^{-1/2}\big)\\
&=\sigma_{\min}^{2}\big(A^{1/2}\big(I-P(P^{T}AP)^{-1}P^{T}A\big)S(S^{T}XS)^{-1/2}\big).
\end{align*}
The desired result follows immediately from~\eqref{ideal_set2}. 
\end{proof}

Let
\begin{equation}\label{def_B}
B=A-AP(P^{T}AP)^{-1}P^{T}A.
\end{equation}
It is not difficult to see that $\lambda_{\min}(A_{X})$ and $\lambda_{\min}(B_{X})$ are the smallest eigenvalues of the generalized eigenvalue problems 
\begin{displaymath}
S^{T}AS\mathbf{v}=\lambda S^{T}XS\mathbf{v} \quad \text{and} \quad S^{T}BS\mathbf{v}=\nu S^{T}XS\mathbf{v},
\end{displaymath}
respectively. Furthermore, the sets $\mathbb{P}_{1}$ and $\mathbb{P}_{2}$ (see~\eqref{eqv_P1} and ~\eqref{eqv_P2}) can be equivalently defined as~\eqref{def_P1} and~\eqref{def_P2} below, respectively.

In what follows, for convenience, we define
\begin{align}
\mathbb{P}_{0}&:=\Big\{P:P^{T}AS=0\Big\},\label{def_P0}\\
\mathbb{P}_{1}&:=\Big\{P: \Null\big(P^{T}AS(S^{T}XS)^{-1/2}\big) \cap \big\{\mathbf{v}\in\mathbb{R}^{n_{s}}\backslash\{0\}:A_{X}\mathbf{v}=\lambda_{\min}(A_{X})\mathbf{v}\big\}\neq\varnothing\Big\},\label{def_P1}\\
\mathbb{P}_{2}&:=\Big\{P:\Null\big(P^{T}AS(S^{T}XS)^{-1/2}\big) \cap \big\{\mathbf{v}\in\mathbb{R}^{n_{s}}\backslash\{0\}:B_{X}\mathbf{v}=\lambda_{\min}(B_{X})\mathbf{v}\big\}\neq\varnothing\Big\}. \label{def_P2}
\end{align}
The following theorem provides sufficient, necessary, and equivalent conditions of the ideal interpolation.

\begin{theorem}\label{main_theo}
Under the condition $(\mathbf{C})$, it holds that
\begin{equation}
\mathbb{P}_{0}\subseteq\mathbb{P}_{2}=\mathbb{P}_{\star}\subseteq\mathbb{P}_{1},
\end{equation}
where $\mathbb{P}_{\star}$, $\mathbb{P}_{0}$, $\mathbb{P}_{1}$, and $\mathbb{P}_{2}$ are given by~\eqref{ideal_set2}, \eqref{def_P0}, \eqref{def_P1}, and~\eqref{def_P2}, respectively.
\end{theorem}

\begin{proof}
(i) ``$\mathbb{P}_{0}\subseteq\mathbb{P}_{2}$'': If $P\in\mathbb{P}_{0}$, then $P^{T}AS=0$ and hence $\Null\big(P^{T}AS(S^{T}XS)^{-1/2}\big)=\mathbb{R}^{n_{s}}$. Obviously, 
\begin{displaymath}
\Null\big(P^{T}AS(S^{T}XS)^{-1/2}\big)\cap\big\{\mathbf{v}\in\mathbb{R}^{n_{s}}\backslash\{0\}:B_{X}\mathbf{v}=\lambda_{\min}(B_{X})\mathbf{v}\big\}\neq\varnothing,
\end{displaymath}
that is, $P\in\mathbb{P}_{2}$, which implies that $\mathbb{P}_{0}\subseteq\mathbb{P}_{2}$.

(ii) ``$\mathbb{P}_{\star}\subseteq\mathbb{P}_{1}$'': From the definitions~\eqref{def_Ax} and~\eqref{def_Bx}, we have that both $A_{X}$ and $B_{X}$ are SPD and $A_{X}-B_{X}$ is symmetric positive semidefinite (SPSD). Hence, it always holds that
\begin{displaymath}
\lambda_{\min}(B_{X})\leq\lambda_{\min}(A_{X}).
\end{displaymath}
Suppose that there exists a $P\in\mathbb{P}_{\star}$ such that $P\notin\mathbb{P}_{1}$. Then
\begin{displaymath}
\Null\big(P^{T}AS(S^{T}XS)^{-1/2}\big)\cap\big\{\mathbf{v}\in\mathbb{R}^{n_{s}}\backslash\{0\}:A_{X}\mathbf{v}=\lambda_{\min}(A_{X})\mathbf{v}\big\}=\varnothing.
\end{displaymath}
Hence, for any $\mathbf{v}_{0}\in\big\{\mathbf{v}\in\mathbb{R}^{n_{s}}\backslash\{0\}:A_{X}\mathbf{v}=\lambda_{\min}(A_{X})\mathbf{v}\big\}$, we have
\begin{displaymath}
\lambda_{\min}(A_{X})=\frac{\mathbf{v}_{0}^{T}A_{X}\mathbf{v}_{0}}{\mathbf{v}_{0}^{T}\mathbf{v}_{0}}>\frac{\mathbf{v}_{0}^{T}B_{X}\mathbf{v}_{0}}{\mathbf{v}_{0}^{T}\mathbf{v}_{0}}\geq\lambda_{\min}(B_{X}).
\end{displaymath}
According to~\eqref{ideal_set2}, we deduce that 
$P\notin\mathbb{P}_{\star}$, which is a contradiction. In other words, for any $P\in\mathbb{P}_{\star}$, we have $P\in\mathbb{P}_{1}$, which yields $\mathbb{P}_{\star}\subseteq\mathbb{P}_{1}$.

(iii) ``$\mathbb{P}_{2}=\mathbb{P}_{\star}$'': If $P\in\mathbb{P}_{2}$, then
\begin{displaymath}
\Null\big(P^{T}AS(S^{T}XS)^{-1/2}\big)\cap\big\{\mathbf{v}\in\mathbb{R}^{n_{s}}\backslash\{0\}:B_{X}\mathbf{v}=\lambda_{\min}(B_{X})\mathbf{v}\big\}\neq\varnothing.
\end{displaymath}
Hence, for any $\mathbf{v}_{1}\in\Null\big(P^{T}AS(S^{T}XS)^{-1/2}\big)\cap\big\{\mathbf{v}\in\mathbb{R}^{n_{s}}\backslash\{0\}:B_{X}\mathbf{v}=\lambda_{\min}(B_{X})\mathbf{v}\big\}$, we have
\begin{displaymath}
A_{X}\mathbf{v}_{1}=B_{X}\mathbf{v}_{1}=\lambda_{\min}(B_{X})\mathbf{v}_{1}.
\end{displaymath}
This shows that $\lambda_{\min}(B_{X})$ is an eigenvalue of $A_{X}$ and hence
\begin{displaymath}
\lambda_{\min}(B_{X})\geq\lambda_{\min}(A_{X}).
\end{displaymath}
Because $\lambda_{\min}(B_{X})\leq\lambda_{\min}(A_{X})$, we get from~\eqref{ideal_set2} that $P\in\mathbb{P}_{\star}$, which yields $\mathbb{P}_{2}\subseteq\mathbb{P}_{\star}$.

On the other hand, if $P\in\mathbb{P}_{\star}$, then
\begin{displaymath}
\lambda_{\min}(B_{X})=\lambda_{\min}(A_{X}).
\end{displaymath} 
Since $\mathbb{P}_{\star}\subseteq\mathbb{P}_{1}$, we obtain
\begin{displaymath}
\Null\big(P^{T}AS(S^{T}XS)^{-1/2}\big)\cap\big\{\mathbf{v}\in\mathbb{R}^{n_{s}}\backslash\{0\}:A_{X}\mathbf{v}=\lambda_{\min}(A_{X})\mathbf{v}\big\}\neq\varnothing.
\end{displaymath}
Then, for any $\mathbf{v}_{2}\in\Null\big(P^{T}AS(S^{T}XS)^{-1/2}\big)\cap\big\{\mathbf{v}\in\mathbb{R}^{n_{s}}\backslash\{0\}:A_{X}\mathbf{v}=\lambda_{\min}(A_{X})\mathbf{v}\big\}$, we have
\begin{displaymath}
B_{X}\mathbf{v}_{2}=A_{X}\mathbf{v}_{2}=\lambda_{\min}(A_{X})\mathbf{v}_{2}=\lambda_{\min}(B_{X})\mathbf{v}_{2},
\end{displaymath}
which yields 
\begin{displaymath}
\mathbf{v}_{2}\in\Null\big(P^{T}AS(S^{T}XS)^{-1/2}\big)\cap\big\{\mathbf{v}\in\mathbb{R}^{n_{s}}\backslash\{0\}:B_{X}\mathbf{v}=\lambda_{\min}(B_{X})\mathbf{v}\big\}.
\end{displaymath}
That is, $P\in\mathbb{P}_{2}$, which yields $\mathbb{P}_{\star}\subseteq\mathbb{P}_{2}$. This completes the proof.
\end{proof}

It is well-known that a successful TG (or MG) algorithm should establish a balance between the smoother $M$ and the coarse-space $\Range(P)$. The definition of $\mathbb{P}_{2}$ (or, equivalently, $\mathbb{P}_{\star}$) has reflected such a wisdom. That is to say, we should  take the smoother $M$ (noting that $X$ typically relies on $M$) into account in order to select an ideal interpolation $P$.

If $P^{T}AS$ is of full column rank, then $\Null(P^{T}AS)=\{0\}$,
which implies
\begin{displaymath}
\Null(P^{T}AS)\cap\big\{\mathbf{v}\in\mathbb{R}^{n_{s}}\backslash\{0\}:S^{T}BS\mathbf{v}=\lambda_{\min}(B_{X})S^{T}XS\mathbf{v}\big\}=\varnothing.
\end{displaymath}
Hence, if $P$ is an ideal interpolation, then $P^{T}AS$ cannot have full column rank. On the basis of this observation and Theorem~\ref{main_theo}, we can obtain the following corollary.

\begin{corollary}\label{R_T}
Assume that $RAS\in\mathbb{R}^{n_{c}\times n_{s}}$ is not of full column rank and
\begin{displaymath}
S^{T}XS=\alpha\big(S^{T}AS-S^{T}AR^{T}(RAR^{T})^{-1}RAS\big)
\end{displaymath}
for some $\alpha>0$. Then $P=R^{T}$ is an ideal interpolation.
\end{corollary}

\begin{proof}
The proof is straightforward by using $\mathbb{P}_{\star}=\mathbb{P}_{2}$ and the definition of $\mathbb{P}_{2}$.
\end{proof}

Traditionally, to define a TG method, the smoother $M$ is pre-selected to provide an $A$-convergent iterative method, such as weighted Jacobi,  Gauss--Seidel, incomplete factorization, overlapping Schwarz methods, etc. Thus, the main task of a TG method is to construct a ``good'' interpolation $P$. On the other hand, for a given interpolation $P$, we can select an appropriate smoother $M$ to ensure that $P$ is ideal; see the following remark for an example.

\begin{remark}\label{sparse_P}\rm
Let $A$ be of the two-by-two block form
\begin{displaymath}
A=\begin{pmatrix}
A_\text{ff} & A_\text{fc}\\
A_\text{cf} & A_\text{cc}
\end{pmatrix},
\end{displaymath}
where $A_\text{ff}\in\mathbb{R}^{n_{s}\times n_{s}}$ $(n_{s}=n-n_{c})$ and $A_\text{cc}\in\mathbb{R}^{n_{c}\times n_{c}}$ with $n_{s}>n_{c}$. Taking
\begin{displaymath}
R=\begin{pmatrix}
0 & I_{n_{c}}
\end{pmatrix} \quad \text{and} \quad S=\begin{pmatrix}
I_{n_{s}}\\
0
\end{pmatrix} ,
\end{displaymath}
we then have
\begin{displaymath}
S^{T}AS-S^{T}AR^{T}(RAR^{T})^{-1}RAS=A_\text{ff}-A_\text{fc}A_\text{cc}^{-1}A_\text{cf},
\end{displaymath}
which is the Schur complement of $A_\text{cc}$ in $A$. Obviously, $RAS$ is not of full column rank due to $n_{s}>n_{c}$. Choosing
\begin{displaymath}
X=\begin{pmatrix}
\alpha\big(A_\text{ff}-A_\text{fc}A_\text{cc}^{-1}A_\text{cf}\big) & \ast\\
\ast & \ast
\end{pmatrix},
\end{displaymath}
we deduce from Corollary~\ref{R_T} that
\begin{displaymath}
P=\begin{pmatrix}
0\\
I_{n_{c}}
\end{pmatrix}
\end{displaymath}
is an ideal interpolation. If we set $X=\frac{1}{2}(M+M^{T})$, then $M$ can be chosen as the following forms:
\begin{align*}
M_{1}&=\begin{pmatrix}
\alpha_{1}\big(A_\text{ff}-A_\text{fc}A_\text{cc}^{-1}A_\text{cf}\big) & 0\\
0 & \alpha_{1}\diag(A_\text{cc})
\end{pmatrix}, \quad M_{2}=\begin{pmatrix}
\alpha_{2}\big(A_\text{ff}-A_\text{fc}A_\text{cc}^{-1}A_\text{cf}\big) & 0\\
0 & \alpha_{2}A_\text{cc}
\end{pmatrix},\\
M_{3}&=\begin{pmatrix}
\alpha_{3}\big(A_\text{ff}-A_\text{fc}A_\text{cc}^{-1}A_\text{cf}\big) & 0\\
A_\text{cf} & \alpha_{3}\diag(A_\text{cc})
\end{pmatrix}, \quad M_{4}=\begin{pmatrix}
\alpha_{4}\big(A_\text{ff}-A_\text{fc}A_\text{cc}^{-1}A_\text{cf}\big) & 0\\
A_\text{cf} & \alpha_{4}A_\text{cc}
\end{pmatrix}.
\end{align*}
To guarantee the iteration~\eqref{iter} is $A$-convergent, we can select the parameters $\alpha_{i}$ so that $M_{i}+M_{i}^{T}-A$ $(i=1,\ldots,4)$ are SPD's.
\end{remark}

Example~\ref{count_exam} has demonstrated that the set $\mathbb{P}_{\star}\backslash\mathbb{P}_{0}$ may be nonempty, even if $X$ is a scalar matrix or the diagonal of $A$. However, if $S^{T}XS=\alpha S^{T}AS$ for some $\alpha>0$, it holds that $\mathbb{P}_{\star}=\mathbb{P}_{0}$ (i.e., an interpolation $P_{\star}$ is ideal if and only if $P_{\star}^{T}AS=0$), which is proved in the following theorem.

\begin{theorem}\label{add_cond}
Under the condition $(\mathbf{C})$, if $S^{T}XS=\alpha S^{T}AS$ for some $\alpha>0$, then
$\mathbb{P}_{\star}=\mathbb{P}_{0}$.
\end{theorem}

\begin{proof}
Without loss of generality, we assume that $\alpha=1$. If $S^{T}XS=S^{T}AS$, then
\begin{displaymath}
A_{X}=I \quad \text{and} \quad B_{X}=I-(S^{T}AS)^{-1/2}S^{T}AP(P^{T}AP)^{-1}P^{T}AS(S^{T}AS)^{-1/2}.
\end{displaymath}
In this case, from~\eqref{ideal_set2} we have that $\mathbb{P}_{\star}$ can be expressed as
\begin{displaymath}
\mathbb{P}_{\star}=\Big\{P:\lambda_{\max}\big((S^{T}AS)^{-1/2}S^{T}AP(P^{T}AP)^{-1}P^{T}AS(S^{T}AS)^{-1/2}\big)=0\Big\}.
\end{displaymath}
Note that $(S^{T}AS)^{-1/2}S^{T}AP(P^{T}AP)^{-1}P^{T}AS(S^{T}AS)^{-1/2}$ is SPSD. Hence, if $P\in\mathbb{P}_{\star}$, then the eigenvalues of $(S^{T}AS)^{-1/2}S^{T}AP(P^{T}AP)^{-1}P^{T}AS(S^{T}AS)^{-1/2}$ are all zero, which implies
\begin{displaymath}
(S^{T}AS)^{-1/2}S^{T}AP(P^{T}AP)^{-1}P^{T}AS(S^{T}AS)^{-1/2}=0.
\end{displaymath}
This shows that $P^{T}AS=0$ (i.e., $P\in\mathbb{P}_{0}$), which yields $\mathbb{P}_{\star}\subseteq\mathbb{P}_{0}$. The desired result follows from the fact $\mathbb{P}_{0}\subseteq\mathbb{P}_{\star}$.
\end{proof}

\begin{remark}\rm
We now give an example to illustrate that the above condition $S^{T}XS=\alpha S^{T}AS$ can be satisfied by choosing appropriate $M$ and $X$. Let $A$ be partitioned as the form 
\begin{displaymath}
A=D+L+L^{T},
\end{displaymath} where $D$ and $L$ denote the diagonal and strictly lower triangular parts of $A$, respectively. For any $\varepsilon>0$, we set
\begin{displaymath}
M=\bigg(\frac{1}{2}+\varepsilon\bigg)D+(1+2\varepsilon)L \quad \text{and} \quad X=\frac{1}{2}(M+M^{T}).
\end{displaymath}
In this case, $M+M^{T}-A=2\varepsilon A$ is SPD and hence the relaxation process~\eqref{iter} is $A$-convergent. And, for any $S$, it holds that
\begin{displaymath} S^{T}XS=\bigg(\frac{1}{2}+\varepsilon\bigg)S^{T}AS.
\end{displaymath}
\end{remark}

\section{A new expression for the ideal interpolation in $\mathbb{P}_{0}$}
\label{sec:Exp}
\setcounter{equation}{0}

In view of Theorem~\ref{main_theo}, we conclude that the condition $P_{\star}^{T}AS=0$ is sufficient to guarantee that $P_{\star}$ is an ideal interpolation, but it is not necessary in general. Example~\ref{count_exam} has shown that the ideal interpolation may not be unique. However, if we attempt to seek the ideal interpolation $P_{\star}$ in $\mathbb{P}_{0}$, then $P_{\star}$ is unique as long as $R$ is fixed. Moreover, $P_{\star}$ has the following explicit expression which does not involve the auxiliary operator $S$ (noting that the measure~\eqref{g_meas} does not involve the operator $S$).

\begin{theorem}\label{main_theo1}
Under the condition $(\mathbf{C})$, the unique ideal interpolation $P_{\star}\in\mathbb{P}_{0}$ can be expressed as 
\begin{equation}\label{new_exp}
P_{\star}=A^{-1}R^{T}(RA^{-1}R^{T})^{-1}.
\end{equation}
\end{theorem}

\begin{proof}
Due to $S^{T}AP_{\star}=0$ and $RP_{\star}=I_{n_{c}}$, it follows that
\begin{displaymath}
\begin{pmatrix}
S^{T}A\\
R
\end{pmatrix}P_{\star}
=\begin{pmatrix}
0\\
I_{n_{c}}
\end{pmatrix}.
\end{displaymath}
It is easy to check that
\begin{displaymath}
\begin{pmatrix}
S^{T}A\\
R
\end{pmatrix}^{-1}=\begin{pmatrix}
S(S^{T}AS)^{-1} & A^{-1}R^{T}(RA^{-1}R^{T})^{-1}
\end{pmatrix}.
\end{displaymath}
Consequently, we arrive at
\begin{displaymath}
P_{\star}=A^{-1}R^{T}(RA^{-1}R^{T})^{-1},
\end{displaymath}
which completes the proof.
\end{proof}

In view of the expression~\eqref{new_exp}, one needs only $A$ and $R$ to compute the ideal interpolation in $\mathbb{P}_{0}$. Using~\eqref{new_exp}, we can derive the same results as in~\cite[Corollaries 3.4 and 3.5]{Falgout2004}.

Let $Q_{\star}=P_{\star}R$, where $P_{\star}$ is given by~\eqref{new_exp}. We then have
\begin{displaymath}
Q_{\star}=A^{-1}R^{T}(RA^{-1}R^{T})^{-1}R,
\end{displaymath}
which is an $A$-orthogonal projection onto $\Range(A^{-1}R^{T})$. Hence, the following corollary holds.

\begin{corollary}
Let $Q_{\star}=P_{\star}R$, where $P_{\star}$ is given by~\eqref{new_exp}. Then
\begin{displaymath}
\|Q\|_{A}\geq\|Q_{\star}\|_{A}=1 \quad \text{and} \quad \|I-Q\|_{A}\geq\|I-Q_{\star}\|_{A}=1.
\end{displaymath}
\end{corollary}

\begin{remark}\label{equiv_P0}\rm
We remark that $P^{T}AS=0$ is equivalent to $R=(P^{T}AP)^{-1}P^{T}A$. In fact, if $P^{T}AS=0$, then
\begin{displaymath}
\Range(AP)=\Range(R^{T}), 
\end{displaymath}
since $RS=0$ and $\rank(AP)=\rank(R^{T})$. Hence, there exists a nonsingular matrix $Z\in\mathbb{R}^{n_{c}\times n_{c}}$ such that $R=ZP^{T}A$.
Using $RP=I_{n_{c}}$, we obtain that $Z=(P^{T}AP)^{-1}$. Thus,  
\begin{displaymath}
R=(P^{T}AP)^{-1}P^{T}A.
\end{displaymath}
Conversely, if $R=(P^{T}AP)^{-1}P^{T}A$, we deduce from $RS=0$ that $P^{T}AS=0$. In addition, we can easily see that $P^{T}AS=0$ is also equivalent to $\Range(AP)=\Range(R^{T})$. As a result, we get an equivalent expression for $\mathbb{P}_{0}$, i.e.,
\begin{displaymath}
\mathbb{P}_{0}=\Big\{P:\Range(AP)=\Range(R^{T})\Big\},
\end{displaymath}
which does not involve the auxiliary operator $S$ as well.
\end{remark}

In view of Remark~\ref{equiv_P0}, if $P_{\star}\in\mathbb{P}_{0}$, we have 
\begin{displaymath}
RA^{-1}=(P_{\star}^{T}AP_{\star})^{-1}P_{\star}^{T}.
\end{displaymath}
By $P_{\star}^{T}R^{T}=I_{n_{c}}$, we have
\begin{displaymath}
P_{\star}^{T}AP_{\star}=(RA^{-1}R^{T})^{-1},
\end{displaymath}
which leads to
\begin{displaymath}
\begin{pmatrix}
S^{T}\\
P_{\star}^{T}
\end{pmatrix}
AP_{\star}=\begin{pmatrix}
0\\
(RA^{-1}R^{T})^{-1}
\end{pmatrix}.
\end{displaymath}
Using~\eqref{inv_SP}, we obtain
\begin{displaymath}
P_{\star}=A^{-1}\begin{pmatrix}
(I-P_{\star}R)^{T}AS(S^{T}AS)^{-1} & R^{T}
\end{pmatrix}\begin{pmatrix}
0\\
(RA^{-1}R^{T})^{-1}
\end{pmatrix}=A^{-1}R^{T}(RA^{-1}R^{T})^{-1}.
\end{displaymath}
This serves as an alternative proof of~\eqref{new_exp}.

By recalling the general expression for $P$ in~\eqref{g_exp}, we can write 
\begin{displaymath}
P_{\star}=R^{T}(RR^{T})^{-1}+S(S^{T}AS)^{-1}S^{T}AY_{\star},
\end{displaymath}
where $Y_{\star}\in\mathbb{R}^{n\times n_{c}}$. Using $S^{T}AP_{\star}=0$, we get
\begin{displaymath}
S^{T}AY_{\star}=-S^{T}AR^{T}(RR^{T})^{-1},
\end{displaymath}
which yields
\begin{equation}\label{old_exp1}
P_{\star}=\big(I-S(S^{T}AS)^{-1}S^{T}A\big)R^{T}(RR^{T})^{-1}.
\end{equation}
Note that~\eqref{old_exp1} coincides with~\eqref{new_exp}. In fact, by $RS=0$ and~\eqref{equiv2}, we have
\begin{displaymath}
A^{-1}R^{T}=\big(I-S(S^{T}AS)^{-1}S^{T}A\big)A^{-1}R^{T}=\big(I-S(S^{T}AS)^{-1}S^{T}A\big)R^{T}(RR^{T})^{-1}RA^{-1}R^{T}.
\end{displaymath}
Therefore, we arrive at
\begin{displaymath}
A^{-1}R^{T}(RA^{-1}R^{T})^{-1}=\big(I-S(S^{T}AS)^{-1}S^{T}A\big)R^{T}(RR^{T})^{-1}.
\end{displaymath}
If $RR^{T}=I_{n_{c}}$, then~\eqref{old_exp1} reduces to~\eqref{old_exp}. We mention that a similar expression of~\eqref{old_exp1} has been given in~\cite{Brannick2017}.

\begin{remark}\rm
Since $S$ and $R^{T}$ form an $L^{2}$-orthogonal decomposition of $\mathbb{R}^{n}$, $\Range(S)$ is unique if $R$ is fixed. However, the operator $S$ itself has different choices. It is not very clear to see whether $P_{\star}$ in~\eqref{old_exp1} is independent of the choice of $S$. On the other hand, the new expression~\eqref{new_exp} explicitly shows that $P_{\star}$ is unique (independent of the choice of $S$) as long as $R$ is fixed.
\end{remark} 

Although $\Range(S)\perp\Range(R^{T})$ with respect to $L^{2}$-inner product and $S(S^{T}AS)^{-1}S^{T}A$ is an $L^{2}$-projection onto $\Range(S)$, the equality $S(S^{T}AS)^{-1}S^{T}AR^{T}(RR^{T})^{-1}=0$ does not hold in general (unless $RAS=0$), because the projection $S(S^{T}AS)^{-1}S^{T}A$ is \emph{oblique} with respect to $L^{2}$-inner product. Hence, under the conditions $RAS\neq0$ and $RR^{T}=I_{n_{c}}$, the ideal interpolation in $\mathbb{P}_{0}$ cannot be of the form $P_{\star}=R^{T}$, while the tentative operator $R^{T}$ could be an ideal choice if we seek the ideal interpolation in $\mathbb{P}_{\star}\backslash\mathbb{P}_{0}$ instead (see Example~\ref{count_exam} and Corollary~\ref{R_T}). 

In view of the expression~\eqref{old_exp1}, we observe that the ideal interpolation in $\mathbb{P}_{0}$ is typically dense. In practice, we would like to have a sparse coarse-grid matrix $A_{c}=P^{T}AP$, which imposes the requirement on $P$ to be sparse as well. According to Corollary~\ref{R_T} and Remark~\ref{sparse_P}, we deduce that it is possible to find a sparse ideal interpolation in $\mathbb{P}_{\star}\backslash\mathbb{P}_{0}$.

\section{Conclusions}
\label{sec:Con}

In this paper, we have established sufficient, necessary, and equivalent conditions of the ideal interpolation in AMG methods. Our result suggests that one has more room than $\mathbb{P}_{0}$ to construct an ideal interpolation. Furthermore, we have derived a new expression for the ideal interpolation in $\mathbb{P}_{0}$, which does not involve the operator $S$. Designing new AMG algorithms based on our result is an interesting topic that deserves in-depth study in the future.

\section*{Acknowledgements}

The authors would like to thank the anonymous referees for their valuable comments and suggestions, which greatly improved the original version of this paper. This work was supported by the National Key Research and Development Program of China (Grant No. 2016YFB0201304), the Major Research Plan of National Natural Science Foundation of China (Grant Nos. 91430215, 91530323), and the Key Research Program of Frontier Sciences of CAS. 

\bibliographystyle{abbrv}
\bibliography{references}

\end{document}